\documentclass[11pt, a4paper]{article}
\usepackage{t1enc}
\usepackage[latin1]{inputenc}
\usepackage[english]{babel}
\usepackage{amsmath,amsthm}
\usepackage{amsfonts}
\usepackage{latexsym}
\usepackage[dvips]{graphicx}
\usepackage{float}
\usepackage[natural]{xcolor}
\usepackage{algorithm}
\usepackage{algorithmic}
\usepackage{enumerate}
\usepackage{multirow}
\usepackage{xcolor}
\usepackage[colorlinks,linkcolor=blue]{hyperref}
\usepackage{lineno}
\usepackage{setspace}
\usepackage{fullpage}
\usepackage[title]{appendix}
\usepackage{amssymb}

\usepackage{listings}
\newcommand{\eps}{\varepsilon}

\newcommand{\de}{\delta}

\theoremstyle{plain}
\newtheorem{thm}{Theorem}[section]
\newtheorem{theorem}[thm]{Theorem}
\newtheorem{conjecture}[thm]{Conjecture}
\newtheorem{lemma}[thm]{Lemma}
\newtheorem{corollary}[thm]{Corollary}
\newtheorem{proposition}[thm]{Proposition}

\theoremstyle{definition}

\newtheorem{question}[thm]{Question}
\newtheorem{problem}[thm]{Problem}
\newtheorem{remark}[thm]{Remark}

\newtheorem{definition}[thm]{Definition}
\newtheorem{claim}[thm]{Claim}
\newtheorem{fact}[thm]{Fact}

\newtheorem{example}[thm]{Example}

\newtheorem{defn-thm}[thm]{Definition-Theorem}

\newcommand{\btheorem}{\begin{theorem}}
\newcommand{\etheorem}{\end{theorem}}
\newcommand{\bconjecture}{\begin{conjecture}}
\newcommand{\econjecture}{\end{conjecture}}
\newcommand{\bproposition}{\begin{proposition}}
\newcommand{\eproposition}{\end{proposition}}
\newcommand{\bdefinition}{\begin{definition}}
\newcommand{\edefinition}{\end{definition}}
\newcommand{\bcorollary}{\begin{corollary}}
\newcommand{\ecorollary}{\end{corollary}}
\newcommand{\bproof}{\begin{proof}}
\newcommand{\eproof}{\end{proof}}
\newcommand{\bclaim}{\begin{claim}}
\newcommand{\eclaim}{\end{claim}}
\newcommand{\bquestion}{\begin{question}}
\newcommand{\equestion}{\end{question}}
\newcommand{\bfact}{\begin{fact}}
\newcommand{\efact}{\end{fact}}
\newcommand{\bremark}{\begin{remark}}
\newcommand{\eremark}{\end{remark}}
\newcommand{\eexample}{\end{example}}
\newcommand{\bexample}{\begin{example}}

\newcommand{\elemma}{\end{lemma}}
\newcommand{\blemma}{\begin{lemma}}

\begin{document}

\title{Clique-factors in graphs with sublinear $\ell$-independence number}

\author{
Jie Han\thanks{School of Mathematics and Statistics and Center for Applied Mathematics, Beijing Institute of Technology, Beijing, China. Email: {\tt han.jie@bit.edu.cn.}}
 \and
Ping Hu\thanks{School of Mathematics, Sun Yat-sen University, Guangzhou, China. Email: {\tt huping9@mail.sysu.edu.cn.} Supported in part by National Key Research and Development Program of China (2021YFA1002100) and National Natural Science Foundation of China (11931002).}
 \and
Guanghui Wang\thanks{School of Mathematics, Shandong University, Jinan, China, Email: {\tt ghwang@sdu.edu.cn}. Research supported by Natural Science
Foundation of China (11871311,12231018) and Young Taishan Scholars probgram of Shandong Province( 201909001).}
 \and
Donglei Yang\thanks{Data Science Institute, Shandong University, Shandong, China. Email: {\tt dlyang@sdu.edu.cn}. Supported by the China Postdoctoral Science Foundation (2021T140413), Natural Science Foundation of China (12101365) and Natural Science Foundation of Shandong Province (ZR2021QA029).}
}

\maketitle
\begin{abstract}
Given a graph $G$ and an integer $\ell\ge 2$, we denote by $\alpha_{\ell}(G)$ the maximum size of a $K_{\ell}$-free subset of vertices in $V(G)$. A recent question of Nenadov and Pehova asks for determining the best possible minimum degree conditions forcing clique-factors in $n$-vertex graphs $G$ with $\alpha_{\ell}(G) = o(n)$, which can be seen as a Ramsey--Tur\'an variant of the celebrated Hajnal--Szemer\'edi theorem. In this paper we find the asymptotical sharp minimum degree threshold for $K_r$-factors in $n$-vertex graphs $G$ with $\alpha_\ell(G)=n^{1-o(1)}$ for all $r\ge \ell\ge 2$.

\end{abstract}

\section{Introduction}

Let $H$ be an $h$-vertex graph and $G$ be an $n$-vertex graph. An \emph{$H$-tiling} is a collection of vertex-disjoint copies of $H$ in $G$. An \emph{$H$-factor} is an $H$-tiling which covers all vertices of $G$. The celebrated Hajnal--Szemer\'{e}di theorem \cite{hajnal1970proof} states that for all integers $n,r$ with $r\ge2$ and $r|n$, any $n$-vertex graph $G$ with $\de(G)\ge (1-\frac{1}{r})n$ contains a $K_r$-factor.
Since then there have been many developments in several directions. From the insight of equitable coloring, Kierstead and Kostochka proved the Hajnal--Szemer\'{e}di theorem with an Ore-type
degree condition \cite{kier2008}. For a general graph $H$, Alon and Yuster \cite{alon1996h} first gave an asymptotic result by showing that if $\delta(G)\ge \left(1-\frac{1}{\chi(H)}\right)n+o(n)$, then $G$ contains an $H$-factor, where $\chi(H)$ is the chromatic number of $H$. Later, K\"{u}hn and Osthus \cite{D2009The} managed to characterise, up to an additive constant, the minimum
degree condition that forces an $H$-factor.
There are also several significant generalisations in the setting of partite graphs \cite{keevash2015multipartite}, directed graphs \cite{Treglown2015} and hypergraphs \cite{rodl2009perfect}.

\subsection{Motivation}
Erd\H{o}s and S\'{o}s~\cite{1970More} initiated the study of a variant of Tur\'{a}n problem which excludes all graphs with large independence number. More generally, for an integer $\ell\ge2$ and a graph $G$, the $\ell$-\emph{independence number} of $G$, denoted by $\alpha_{\ell}(G)$, is the maximum size of a $K_{\ell}$-free subset of vertices. Given integers $n,r$ and a function $f(n)$, we use $\textbf{RT}_{\ell}(n, K_r, f(n))$ to denote the maximum number of edges of an $n$-vertex $K_r$-free graph $G$ with $\alpha_{\ell}(G) \le f(n)$. In particular, the \emph{Ramsey--Tur\'an density} of $K_r$ is defined as $\varrho_{\ell}(r):=\lim\limits_{\alpha\to0}\lim\limits_{n\to\infty}\frac{\textbf{RT}_{\ell}(n, K_r, \alpha n)}{\binom{n}{2}}$. Szemer\'edi~\cite{E1972On} first showed that $\varrho_{2}(4)\le\frac{1}{4}$. This turned out to be sharp as Bollob\'as and Erd\H{o}s~\cite{1976On} provided a matching lower bound using an ingenious geometric construction. There are some recent exciting developments in this area~\cite{BaloghRT2012,BaloghRT2013,fox15,kkl19,Liu2021,Lud19}. For further information on Ramsey--Tur\'an theory the reader is referred to a comprehensive survey~\cite{2001Ramsey} by Simonovits and S\'os.

Note that the extremal example that achieves the optimality of the bound on $\de(G)$ in the Hajnal--Szemer\'{e}di theorem also has large independence number \cite{Balogh2016Triangle}, which makes it far from being typical. Following the spirit of the Ramsey--Tur\'{a}n theory, a natural question on the Hajnal--Szemer\'{e}di theorem is whether the minimum degree condition can be weakened when the host graph has sublinear independence number. The following Ramsey--Tur\'{a}n type problem was proposed by Balogh, Molla and Sharifzadeh \cite{Balogh2016Triangle}.

\begin{problem}\cite{Balogh2016Triangle}\label{p1}
  Let $r\ge3$ be an integer and $G$ be an $n$-vertex graph with $\alpha(G)=o(n)$. What is the minimum degree condition on $G$ that guarantees a $K_r$-factor?
\end{problem}

Balogh, Molla and Sharifzadeh~\cite{Balogh2016Triangle} studied $K_3$-factors and showed that if
the independence number of an $n$-vertex graph $G$ is $o(n)$ and $\de(G)\ge\frac{n}{2}+\eps n$ for any $\eps>0$, then
$G$ contains a triangle factor. Recently Knierim and Su \cite{knierim2019kr} resolved the case $r\ge 4$ by determining the asymptotically tight minimum degree bound $(1- \frac{2}{r})n + o(n)$.\medskip

The following problem was proposed by Nenadov and Pehova~\cite{Nenadov2018}.

\begin{problem}\label{prob1.3}
For all $r,\ell \in \mathbb{N}$ with $r \ge \ell \ge 2$, let $G$ be an $n$-vertex graph with $n\in r\mathbb{N}$ and $\alpha_{\ell}(G) = o(n)$. What is the best possible minimum degree condition on $G$ that guarantees a $K_{r}$-factor?
\end{problem}

Nenadov and Pehova~\cite{Nenadov2018} also provided upper and lower bounds on the minimum degree condition.
In particular, they solved Problem \ref{prob1.3} for $r=\ell+1$ and proved that $n/2+o(n)$ is the correct minimum degree threshold.
Knierim and Su~\cite{knierim2019kr} reiterated Problem~\ref{prob1.3} in their paper and proposed a minimum degree condition as follows.



\begin{problem}\label{prob1.5}\cite{knierim2019kr,Nenadov2018}
Is it true that for every $r,\ell\in\mathbb{N}$ with $r\ge \ell\ge 2$ and $\mu>0$, there exists $\alpha>0$ such that for sufficiently large $n\in r\mathbb{N}$, every $n$-vertex graph $G$ with \[\delta(G) \ge \max \left\{\frac{r - \ell}{r}+\mu, \frac{1}{2}+\mu\right\}n~\text{and}~ \alpha_{\ell}(G) \le \alpha n \] contains a $K_{r}$-factor?
\end{problem}

Very recently, Chang, Han, Kim, Wang and Yang~\cite{chang2021} determines the asymptotically optimal minimum degree condition for $\ell \ge \frac{3}{4}r$, which solves Problem~\ref{prob1.3} for this range, and indeed provides a negative answer to Problem~\ref{prob1.5}.
\btheorem\label{main thm}\emph{\cite{chang2021}}
Let $r,\ell \in \mathbb{N}$ such that $r > \ell \ge \frac{3}{4}r$. For any $\mu >0$, there exists $\alpha>0$ such that for sufficiently large  $n\in r\mathbb{N}$, every $n$-vertex graph $G$ with \[\delta(G) \ge \left( \frac{1}{2-\varrho_{\ell}(r-1)} + \mu \right)n~\text{and}~ \alpha_{\ell}(G) \le \alpha n\]  contains a $K_{r}$-factor.
Moreover, the minimum degree condition is asymptotically best possible.
\etheorem
Based on this result, Problem~\ref{prob1.5} should be revised as follows.

\begin{problem}\label{prob2}
Is it true that $\delta(G) \ge \max \{\frac{r - \ell}{r}+\mu, \frac{1}{2-\varrho_{\ell}(r-1)}+\mu \}n$ suffices in Problem \ref{prob1.5}?
\end{problem}

\subsection{Main results and discussions}

By the aformentioned results, Problem~\ref{prob2} is solved for $\ell=2$ \cite{knierim2019kr} and for $\ell \ge 3r/4$ \cite{chang2021}, both done by quite involved proofs. It seems to us that a complete resolution of Problem~\ref{prob2} is a quite challenging task.

The purpose of this paper is to extend the discussion on the problem to sublinear $\ell$-independence numbers. We first state a simplified version of our main result which says that the answer to Problem~\ref{prob2} is yes if we assume a slightly stronger assumption $\alpha_{\ell}(G) \le n^{1-o(1)}$.
\btheorem\label{th2-}
For $\mu, c\in(0,1)$, $r,\ell \in \mathbb{N}$ such that $r > \ell \ge 2$, the following holds for sufficiently large $n\in r\mathbb{N}$.
Every $n$-vertex graph $G$ with $\delta(G) \ge \max\{\frac{r-\ell}{r}+\mu, \frac{1}{2-\varrho_{\ell}(r-1)} + \mu \}n $ and $\alpha_{\ell}(G) \le n^c$  contains a $K_{r}$-factor.
\etheorem
To state our main result, we define more general versions of the Ramsey--Tur\'an densities as follows.

\bdefinition
Let $f(n)$ be a monotone increasing function, and $r,\ell \in \mathbb{N},\alpha\in(0,1)$.
\begin{enumerate}
	\item[($i$)] Let $\textbf{RT}_{\ell}(n,K_{r},f(\alpha n))$ be the maximum integer $m$ for which there exists an $n$-vertex $K_{r}$-free graph $G$ with $e(G)=m$ and $\alpha_{\ell}(G)\le f(\alpha n)$. Then let
\[
\varrho_\ell(r,f):=\lim\limits_{\alpha\to 0}\limsup\limits_{n\to \infty}\frac{\textbf{RT}_{\ell}(n,K_{r},f(\alpha n))}{\binom n2}.
\]
	\item[($ii$)] Let $\textbf{RT}^*_{\ell}(n,K_{r},f(\alpha n))$ be the maximum integer $\delta$ for which there exists an $n$-vertex $K_{r}$-free graph $G$ with $\delta(G)=\delta$ and $\alpha_{\ell}(G)\le f(\alpha n)$. Then let
\[
\varrho^*_\ell(r,f):=\lim\limits_{\alpha\to 0}\limsup\limits_{n\to \infty}\frac{\textbf{RT}^*_{\ell}(n,K_{r},f(\alpha n))}{n}.
\]
\end{enumerate}
\edefinition
By definition trivially it holds that $\varrho^*_\ell(r,f) \le \varrho_\ell(r,f)$.
It is proved in~\cite{chang2021} that $\varrho^*_\ell(r,f) = \varrho_\ell(r,f)$ if there exists $c\in (0,1)$ such that $xf(n)\le f(x^{1/c} n)$ for every $x\in (0,1)$ and $n\in \mathbb N$.

The full version of our result is stated as follows.

\btheorem\label{th2}
For $\mu >0$, $r,\ell \in \mathbb{N}$ such that $r > \ell \ge 2$, there exists $\alpha>0$ such that the following holds for sufficiently large $n\in r\mathbb{N}$. Let $\lambda=1/\lfloor\frac{r}{\ell}+1\rfloor$ and $f(n)\le n^{1-\omega(n)\log^{-\lambda}n}$ be a monotone increasing function, where $\omega(n)\rightarrow\infty$ slowly.\footnote{The strange-looking function $n^{1-\omega(n)\log^{-\lambda}n}$ satisfies $n^{1-\varepsilon}<n^{1-\omega(n)\log^{-\lambda}n}<\frac{n}{\log n}$ for any constant $\varepsilon>0$.}
Then every $n$-vertex graph $G$ with \[\delta(G) \ge \max\left\{\frac{r-\ell}{r}+\mu, \frac{1}{2-\varrho^*_{\ell}(r-1,f)} + \mu \right\}n ~\text{and}~\alpha_{\ell}(G) < f(\alpha n)\]  contains a $K_{r}$-factor.
\
\etheorem

In fact, Theorem~\ref{th2} implies Theorem~\ref{th2-} by the observation that $\varrho^*_{\ell}(r-1,f)\le \varrho^*_{\ell}(r-1)=\varrho_{\ell}(r-1)$ holds for any function $f(n)=n^c$ with $c\in(0,1)$.
We shall supply lower bound constructions (see next subsection) which show that the minimum degree condition in Theorem~\ref{th2} is asymptotically best possible for $\alpha_{\ell}(G)\in (n^{1-\gamma}, n^{1-\eps})$, for any $\gamma\in(0,\frac{\ell-1}{\ell^2+2\ell})$ and any $\eps>0$.
This can be seen as a stepping stone towards a full understanding of the Ramsey--Tur\'an tiling thresholds for cliques where $\alpha_\ell(G)\in [2, o(n)]$.

Here we also provide some concrete thresholds that we could spell out from Theorem~\ref{th2}.
Recall that Nenadov and Pehova~\cite{Nenadov2018} solved Problem \ref{prob1.3} for $r=\ell+1$ whilst Knierim and Su \cite{knierim2019kr} solved the case $\ell=2, r\ge 4$. Now we consider the first open case $r=\ell+2,\ell\ge 3$. Note that $\varrho^*_{\ell}(\ell+1,f)=\varrho_{\ell}(\ell+1)=0$.
Then Theorem~\ref{main thm} says that for $\ell\ge 6$ (that is, $\ell\ge \frac{3}{4}(\ell+2)$) and $\alpha_{\ell}(G) = o(n)$, $\frac{n}{2}+o(n)$ is the minimum degree threshold forcing a $K_{\ell+2}$-factor. For the remaining cases $\ell\in\{3,4,5\}$, Theorem~\ref{th2} implies that the minimum degree threshold is $\frac{n}{2}+o(n)$ under a stronger condition $\alpha_{\ell}(G) \le n^{1-\eps}$ for any fixed $\eps>0$.

Our proof of Theorem~\ref{th2} uses the absorption method and the regularity method. In particular, we use dependent random choice for embedding cliques in regular tuples with sublinear independence number, which is closely related to the Ramsey--Tur\'an problem.

\subsection{Sharpness of the minimum degree condition}
We note that both terms in the minimum degree condition in Theorem~\ref{th2} are asymptotically best possible.
First, we show that the first term cannot be weakened when $\alpha_{\ell}(G)\in (n^{1-\gamma},o(n))$ for some constant $\gamma$ as follows.
\begin{proposition}\label{lower}
Given integers $r,\ell \in \mathbb{N}$ with $r > \ell \ge 2$ and constants $\eta,\gamma$ with $\eta\in(0,\frac{r-\ell}{r})$, $\gamma\in(0,\frac{\ell-1}{\ell^2+2\ell})$, the following holds for all sufficiently large $n\in\mathbb{N}$ and constant $\mu:=\tfrac{r}{r-\ell}(\tfrac{r-\ell}{r}-\eta)$. There exists an $n$-vertex graph $G$ with $\delta(G)\ge \eta n$ and $\alpha_{\ell}(G)< n^{1-\gamma}$ such that every $K_r$-tiling in $G$ covers at most $(1-\mu)n$ vertices.
\end{proposition}
Indeed, Proposition~\ref{lower} also gives, in the setting that $\alpha_{\ell}(G)\le n^{1-o(1)}$, a lower bound construction for the minimum degree condition forcing an almost $K_r$-tiling that leaves a constant number of vertices uncovered. More results on almost graph tilings can be found in a recent comprehensive paper~\cite{han2021ramseyturan}.

The second term $\frac{1}{2-\varrho^*_{\ell}(r-1,f)}$ is also asymptotically tight, which is given by a cover threshold construction as follows.\\ [5pt]
\textbf{Cover Threshold.}
To have a $K_r$-factor in $G$, a naive necessary condition is that every vertex $v\in V(G)$ is covered by a copy of $K_r$ in $G$. The cover threshold has been first discussed in~\cite{Han2017} and appeared in a few different contexts~\cite{chang2021,chang2020factors,Sun2021quasirandom}.

Now we give a construction that shows the optimality of the term $\frac{1}{2-\varrho^*_{\ell}(r-1,f)}$ for the function $f(n)$ as in Theorem~\ref{th2}. A similar construction can be found in \cite{chang2021}. Given integers $r,\ell$ and constants $\eps,\alpha, x\in(0,1)$, we construct (for large $n\in r\mathbb{N}$) an $n$-vertex graph $G$ by
\begin{enumerate}
  \item [$(i)$] first fixing a vertex $v$ such that $N(v)=xn$ and $G[N(v)]=:G'$ is a $K_{r-1}$-free subgraph with $\delta(G') \ge \varrho^*_{\ell}(r-1,f)xn-\eps n$ and $\alpha_{\ell}(G')\le f(\alpha xn)$;
  \item [$(ii)$] and then adding a clique of size $n-xn-1$ that is complete to $N(v)$.
\end{enumerate}
There exists no copy of $K_r$ covering $v$ and thus $G$ contains no $K_r$-factor; moreover, by choosing $x=\tfrac{1}{2-\varrho^*_{\ell}(r-1,f)}$, we obtain $\delta(G)\ge\frac{1}{2-\varrho^*_{\ell}(r-1,f)}n-\varepsilon n$ and $\alpha_{\ell}(G)=\alpha_{\ell}(G')\le f(\alpha n)$.

\medskip\noindent\textbf{Notation.} Throughout the paper we follow standard graph-theoretic notation~\cite{Diestel2017}. For a graph $G=(V,E)$, let $v(G)=|V|$ and $e(G)=|E|$. For $U\subseteq V$, $G[U]$ denotes the induced subgraph of $G$ on $U$. The notation $G-U$ is used to denote the induced subgraph after
removing $U$, that is, $G-U:=G[V\setminus U]$. For two subsets $A,B\subseteq V(G)$, we use $e(A,B)$ to denote the number of edges joining $A$ and $B$. Given a vertex $v\in v(G)$ and $X\subseteq V(G)$, denote by $N_X(v)$ the set of neighbors of $v$ in $X$ and let $d_X(v):=|N_X(v)|$. In particular, we write $N_{G}(v)$ for the set of neighbors of $v$ in $G$. We omit the index $G$ if the graph is clear from the context. Given a set $V$ and an integer $k$, we write $\binom{V}{k}$ for the family of all $k$-subsets of $V$. For all integers $a,b$ with $a\le b$, let $[a,b]:=\{i\in\mathbb{Z}:a\le i\le b\}$ and $[a]:=\{1,2,\ldots,a\}$.

When we write $\alpha \ll \beta\ll \gamma$, we always mean that $\alpha, \beta, \gamma$ are constants in $(0,1)$, and $\beta\ll \gamma$ means that there exists $\beta_0=\beta_0(\gamma)$ such that the subsequent arguments hold for all $0<\beta\le \beta_0$.
Hierarchies of other lengths are defined analogously. In the remaining proofs, we always take $\lambda=1/\lfloor\frac{r}{\ell}+1\rfloor$ and $f(n)\le n^{1-\omega(n)\log^{-\lambda}n}$ unless otherwise stated.


\section{Proof strategy and Preliminaries}
Our proof uses the absorption method, pioneered by the work of R\"odl, Ruci\'nski and Szemer\'edi~\cite{rodl2009perfect} on perfect matchings in hypergraphs, though similar ideas already appeared implicitly in previous works, e.g. Krivelevich~\cite{Krivelevich_TF}. 
A key step in the absorption method for $H$-factor problem is to show that for every set of $h:=|V(H)|$ vertices, the host graph $G$ contains $\Omega(n^b)$ $b$-vertex absorbers (to be defined shortly).
However,  as pointed out in~\cite{Balogh2016Triangle}, in our setting this is usually impossible because when we construct the absorbers using the independence number condition, it does not give such a strong counting.
Instead, a much weaker notion has been used in this series of works, that is, we aim to show that \emph{for (almost) every set of $h$ vertices, the host graph $G$ contains $\Omega(n)$ vertex-disjoint absorbers}.
Note that this weak notion of absorbers have been successfully used in our setting~\cite{Nenadov2018, knierim2019kr} and the randomly perturbed setting~\cite{chang2020factors}.

\subsection{The absorption method}


Following typical absorption strategies, our main work is to establish an absorbing set (see Lemma~\ref{absorbing lem}) and find an almost-perfect tiling (see Lemma~\ref{tiling lem}). 
We first introduce the following notions of absorbers and absorbing sets from \cite{Nenadov2018}.

\bdefinition Let $H$ be a graph with $h$ vertices and $G$ be a graph with $n$ vertices.
\begin{enumerate}
	\item We say that a subset $A \subseteq V(G)$ is a $\xi$\emph{-absorbing set} for some $ \xi > 0$ if for every subset $R \subseteq V(G)\setminus A$ with $|R| \le \xi n$ and $|A\cup R|\in h\mathbb{N}$, $G[A \cup R]$ contains an $H$-factor.
	\item Given a subset $S \subseteq V(G)$ of size $h$ and an integer $t$, we say that a subset $A_{S} \subseteq V(G)\setminus S$ is an $(S,t)$\emph{-absorber} if $|A_{S}| \le ht$ and both $G[A_{S}]$ and $G[A_{S} \cup S]$ contain an $H$-factor.
\end{enumerate}
\edefinition


Now we are ready to state our first crucial lemma, whose proof can be found in Section~\ref{sec4}.
\blemma[Absorbing Lemma]\label{absorbing lem}
Given positive integers $r,\ell$ with $r>\ell\ge2$ and constants $\mu,\gamma$ with $0<\gamma<\frac{\mu}{2}$, there exist $\alpha, \xi>0$ such that the following holds for sufficiently large $n\in r\mathbb{N}$. Let $G$ be an $n$-vertex graph with $\delta(G) \ge \max\{\frac{r-\ell}{r}+\mu,\frac{1}{2-\varrho^*_{\ell}(r-1,f)} + \mu \}n $ and $\alpha_{\ell}(G) < f(\alpha n)$. Then $G$ contains a $\xi$-absorbing set $A$ of size at most $\gamma n$.
\elemma

Our second crucial lemma is on almost $K_r$-factor as follows, whose proof will be given in Section~\ref{sec3}.

\blemma[Almost perfect tiling]\label{tiling lem}
Given positive integers $r,\ell$ such that $r > \ell \ge 2$ and positive constants $\mu,\delta$, the following statement holds for sufficiently large $n\in r\mathbb{N}$. Every $n$-vertex graph $G$ with $\delta(G) \ge \left( \frac{r-\ell}{r} + \mu \right)n $ and $\alpha_{\ell}(G) < f(n)$ contains a $K_{r}$-tiling that leaves at most $\delta n $ vertices in $G$ uncovered.
\elemma

Now we are ready to prove Theorem~\ref{th2} using Lemma~\ref{absorbing lem} and Lemma~\ref{tiling lem}.

\bproof[Proof of Theorem~\ref{th2}]

Given any positive integers $\ell,r$ with $r>\ell\ge2$ and a constant $\mu>0$.  Choose $\frac{1}{n}\ll\alpha \ll\delta\ll\xi\ll\gamma\ll\mu$. Let $G$ be an $n$-vertex graph with \[\delta(G)\ge \max\left\{\frac{r-\ell}{r}+\mu,  \frac{1}{2-\varrho^*_{\ell}(r-1,f)} + \mu \right\}n ~\text{and}~ \alpha_{\ell}(G) < f(\alpha n).\]
By Lemma~\ref{absorbing lem} with $\gamma\le\frac{\mu}{2}$, we find a $\xi$-absorbing set $A \subseteq V(G)$ of size at most $\gamma n$ for some $\xi>0$. Let $G_1:=G-A$. Then we have
$
\delta(G_1) \ge \left( \tfrac{r-\ell}{r} + \mu \right)n-\gamma n\ge \left( \tfrac{r-\ell}{r} + \tfrac{\mu}{2} \right)n$.
Therefore by applying Lemma \ref{tiling lem} on $G_1$ with $\delta$, we obtain a $K_{r}$-tiling $\mathcal{M}$ that covers all but a set $R$ of at most $\delta n$ vertices in $G_1$. Since $\delta\ll\xi$, the absorbing property of $A$ implies that $G[A \cup R]$ contains a $K_r$-factor $\mathcal{R}$, which together with $\mathcal{M}$ forms a $K_r$-factor in $G$.
\eproof

\section{Finding almost perfect tilings}\label{sec3}
In this section we address Lemma~\ref{tiling lem}. The proof of Lemma~\ref{tiling lem} uses the regularity method, a tiling result of Koml\'{o}s (Theorem~\ref{Komlos thm}), and dependent random choice (Lemma~\ref{hu}). We shall first give the crucial notion of regularity and then introduce the powerful Szemer\'edi's Regularity Lemma.

\subsection{Regularity}

Given a graph $G$ and a pair $(X, Y)$ of vertex-disjoint subsets in $V(G)$, the \emph{density} of $(X,Y)$ is defined as
\[
d(X,Y) = \frac{e(X,Y)}{|X||Y|}.
\]
For constants $\eps,d >0$, we say that $(X, Y)$ is an $\eps$-\emph{regular pair} with density at least $d$ (or $(X, Y)$ is $(\eps, d)$-\emph{regular}) if $d(X,Y)\ge d$ and for all $X' \subseteq X$, $Y' \subseteq Y$ with  $|X'| \ge \eps |X|$, $|Y'| \ge \eps |Y|$, we have
\[
|d(X',Y') - d(X,Y)|  \le  \eps.
\]
Moreover, a pair $(X, Y)$ is called $(\eps,d)$-$super$-$regular$ if $(X, Y)$ is $(\eps,d)$-regular, $d_{Y}(x)\ge d|Y|$ for all $x \in X$ and $d_{X}(y)\ge d|X|$ for all $y \in Y$. The following fact is an easy consequence of the definition of regularity.

\bfact\label{slicing lem}
Given constants $d,\eta >\eps>0$ and a bipartite graph $G =(X \cup Y, E)$, if $(X,Y)$ is $(\eps, d)$-regular, then for all $X_{1} \subseteq X$ and $Y_{1} \subseteq Y$ with $|X_{1}| \ge \eta|X|$ and $|Y_{1}| \ge \eta|Y|$, we have that $(X_{1}, Y_{1})$ is $(\eps',d-\eps)$-regular in $G$ for any $\eps ' \ge {\rm max}\{ \frac{\eps}{\eta},2\eps  \}$. 
\efact

Given a family of vertex-disjoint sets in $V(G)$ which are pairwise $\eps$-regular, we can find in each set a large subset such that every pair of resulting subsets is super-regular.

\begin{proposition}[see Proposition 2.6 in \cite{chang2021}]\label{superregular}
Given a constant $\varepsilon>0$ and integers $m,t$ with $t<\frac{1}{2\varepsilon}$, let $G$ be an $n$-vertex graph and $V_1,V_2,\ldots, V_{t+1}$ be vertex-disjoint subsets each of size $m$ in $G$ such that every pair $(V_i,V_j)$ is $\varepsilon$-regular with density $d_{ij}:=d(V_i,V_j)$. Then there exists for each $i\in[t+1]$ a subset $V'_i\subseteq V_i$ of size at least $(1-t\varepsilon)m$ such that every pair $(V_i',V_j')$ is $(2\varepsilon,d_{ij}-(t+1)\varepsilon)$-super-regular.
\end{proposition}

We now state a degree form of the regularity lemma (see~\cite[Theorem 1.10]{komlos1996}).

\blemma[Degree form of the Regularity Lemma~\cite{komlos1996}]\label{reg}

For every $\eps > 0$ there is an $N = N(\eps )$ such that the following holds for any real number $d \in [0,1]$ and $n\in \mathbb{N}$. Let $G=(V,E)$ be a graph with $n$ vertices. Then there exists a partition $\mathcal{P}=\{V_{0},\ldots, V_{k}\} $ of $V$ and a spanning subgraph $G' \subseteq G$ with the following properties:
\begin{enumerate}
  \item [$(a)$]\label{a} $ \frac{1}{\eps}\le k \le N $;
  \item [$(b)$] $|V_{i}| \le \eps  n $ for $0 \le i \le k$ and $|V_{1}|=|V_{2}|=\cdots=|V_{k}| =m$ for some $m\in \mathbb{N}$;
  \item [$(c)$] $d_{G'}(v) > d_{G}(v) - (d + \eps )n$ for every $v \in V(G)$;
  \item [$(d)$] every $V_{i}$ is an independent set in $G'$;
  \item [$(e)$] each pair $(V_{i},V_{j})$, $1 \le i < j \le k$ is $\eps $-regular in $G'$ with density 0 or at least $d$.
\end{enumerate}
\elemma

A widely-used auxiliary graph accompanied with the regular partition is the reduced graph.
The \emph{$d$-reduced graph} $R_d$ of $\mathcal{P}$ is a graph defined on the vertex set $\{V_1,\ldots,V_k\}$ such that $V_i$ is connected to $V_j$ by an edge if $(V_i,V_j)$ has density at least $d$ in $G'$. So if $V_i$ is not connected to $V_j$, then $(V_{i},V_{j})$ has density $0$ by property $(e)$ above.
To ease the notation, we use $d_R(V_i)$ to denote the degree of $V_i$ in $R_d$ for each $i\in[k]$. Note that $R_d$ also can be regarded as a weighted graph in which the weight for each edge $V_iV_j$, denoted by $d_{ij}$ for simplicity, is exactly the density of the pair $(V_i,V_j)$ in $G'$. %
\begin{fact}\label{min degree}
For positive constants $d,\eps$ and $c$, let $G=(V,E)$ be a graph on $n$ vertices with $\delta(G) \ge cn$. Let $G'$ and $\mathcal{P}$ be obtained by applying Lemma~\ref{reg} on $G$ with constants $d$ and $\eps$. Let $R_d$ be the $d$-reduced graph as given above. Then for every $V_i\in V(R_d)$ we have
$d_{R}(V_i)\ge (c -2\eps  -d)k.$
\end{fact}
\bproof
Note that $|V_0|\le \eps n$ and $|V_i|=m$ for each $V_i\in V(R_d)$. Thus we have
\[
\sum_{V_j \sim V_i} d_{ij}|V_i||V_j|=e_{G'}(V_{i},\cup_{j \ne i}V_{j})\ge (\delta(G')-|V_0|)|V_i|\ge\left(c -2\eps  -d \right)nm,
\]
which implies
\[
d_{R}(V_i)=\sum_{V_j \sim V_i} 1\ge\sum_{V_j \sim V_i} d_{ij}\ge \frac{(c -2\eps -d )nm}{m^2}\ge(c -2\eps -d)k.
\]
\eproof

To find an almost perfect $K_r$-tiling, we shall also make use of the following result of Koml\'os~\cite{komlos2000tiling} on graph tilings. Given a graph $H$ on $r$ vertices, the \emph{critical chromatic number} of $H$ is defined as
$\chi_{cr}(H) = \frac{(k-1)r}{r - \sigma},$ where $k=\chi(H)$ and $\sigma =\sigma(H)$ denotes the smallest size of a color class over all $k$-colorings of $H$.

\btheorem[Koml\'os~\cite{komlos2000tiling}]\label{Komlos thm}
Given any graph $H$ and a constant $\gamma>0$, there exists an integer $n_{0}=n_{0}(\gamma, H)$ such that every graph $G$ of order $n \ge n_{0}$ with $\delta(G) \ge \left(1-\frac{1}{\chi_{cr}(H)}\right)n$ contains an $H$-tiling covering all but at most $\gamma n$ vertices.
\etheorem
Based on this result, we will first apply the regularity lemma to $G$ to get a reduced graph $R:=R_d$ for a constant $d>0$, and then apply Theorem~\ref{Komlos thm} to get an $H$-tiling of $R$ covering almost all vertices for a suitably-chosen auxiliary graph $H$.
To get an almost $K_r$-tiling of $G$ from this almost $H$-tiling of $R$, we will use the following lemma 
which says if we can find a $K_q$ in a copy of $H$ in $R$, then we can find a $K_{pq}$ in $G$ under certain conditions on $\alpha_p(G)$. Its proof follows from that of Claim 6.1 in \cite{BALOGH2015148}, where a similar assumption on $\alpha(G)$ (instead of $\alpha_p(G)$) is used. For completeness we include a proof of Lemma~\ref{hu} in the appendix.


\begin{lemma}\label{hu}
Given a constant $d>0$ and integers $p,q\ge 2$, there exist $C,\eps$ such that for any constant $\eta>0$ the following holds for every sufficiently large $n\in \mathbb{N}$ and $g(n):=n^{1-C\log^{-1/q}n}$. Let $G$ be an $n$-vertex graph with $\alpha_p(G) < g(n)$ and $V_1, V_2,\ldots,V_{q}$ be pairwise vertex-disjoint sets of vertices in $G$ with $|V_i|\ge \eta n$ for each $i\in[q]$ and every pair $(V_i,V_j)$ being $(\eps,d)$-regular. Then there exists a copy of $K_{pq}$ in $G$ which contains exactly $p$ vertices in each $V_i$ for $i\in [q]$.
\end{lemma}

\subsection{Proof of Lemma~\ref{tiling lem}}
\bproof[Proof of Lemma~\ref{tiling lem}]
Given $r,\ell\in \mathbb{N}$ such that $r>\ell\ge2$ and $\mu>0,\delta>0$, we choose
\[
\tfrac{1}{n}\ll \eps \ll\mu,\delta,\tfrac{1}{r}.
\]
Let $G$ be an $n$-vertex graph with $\delta(G) \ge \left( \frac{r-\ell}{r} + \mu \right)n $ and $\alpha_{\ell}(G) < f(n)$. By applying Lemma~\ref{reg} on $G$ with constants $\eps>0$ and $d:=\frac{\mu}{4}$, we obtain a partition $\mathcal{P}=\{V_{0},\ldots,V_{k}\}$ for some $\frac{1}{\eps}\le k \le N$ and a spanning subgraph $G' \subseteq G$ with properties (a)-(e) as stated. Let $m:=|V_i|$ for all $i\in[k]$ and $R_d$ be the corresponding $d$-reduced graph of $\mathcal{P}$. Then it follows from Fact~\ref{min degree} that $\delta(R_{d}) \ge (\frac{r-\ell}{r}+\frac{\mu}{4}) k$.

Let $r=x\ell+y$ for some integers $x,y$ with $x\ge 1,1\le y\le \ell$. Note that the complete $(x+1)$-partite graph $H:=K_{y,\ell,\ldots,\ell}$ has $\chi_{cr}(H)=\frac{r}{\ell}$. Now we apply Theorem~\ref{Komlos thm} on $R_{d}$ with $\gamma=\frac{\delta}{2}$ and $H=K_{y,\ell,\ldots,\ell}$ to obtain a family $\mathcal{H}$ of vertex-disjoint copies of $H$ that cover all but at most $\frac{\delta}{2} k$ vertices of $R_{d}$.

Given a copy of $H$ in $\mathcal{H}$, without loss of generality, we may assume that its vertex set is $\{V_{1},\ldots, V_{r}\}$ together with the parts denoted by
\[\mathcal{W}_1=\{V_1,\ldots,V_y\}~ \text{and} ~ \mathcal{W}_{s+1}=\{V_{y+1+(s-1)\ell},\ldots,V_{y+s\ell}\}~\text{for}~s\in [x].\] Note that every pair of clusters $V_i,V_j$ from distinct parts forms an $\eps$-regular pair with density at least $d$.

We shall greedily embed in the original graph $G$ vertex-disjoint copies of $K_r$ that together cover almost all the vertices in $\cup_{i=1}^{r}V_i$.
Now for each $i\in[y]$ we divide $V_{i}$ arbitrarily into $\ell$ subclusters $V_{i,1},\ldots,V_{i,\ell}$ of (almost) equal size. For each $j\in [y+1,r]$ we divide $V_{j}$ into $y$ subclusters $V_{j,1},\ldots,V_{j,y}$ of (almost) equal size. Here for simplicity we may further assume that $|V_{i,i'}|=\frac{m}{\ell}$ for $i\in[y],i'\in[\ell]$ and $|V_{j,j'}|=\frac{m}{y}$ for every $j\in[y+1,r],j'\in[y]$.
We call a family $\{V_{i_s,j_s}\}_{s=1}^{x+1}$ of $x+1$ subclusters \emph{legal} if $V_{i_s}\in \mathcal{W}_s$ for each $s\in[x+1]$, i.e., $\{V_{i_s}\}_{s=1}^{x+1}$ forms a copy of $K_{x+1}$ in $R_d$.
Note that each $\mathcal{W}_s$ ($s\in [x+1]$) contains exactly $y\ell$ subclusters in total. Therefore we can greedily partition the set of all subclusters into $y\ell$ pairwise disjoint legal families.

Now if we have a $K_r$-tiling in $G$ for every legal family $\{V_{i_s,j_s}\}_{s=1}^{x+1}$, that covers all but at most $\frac{r\delta}{4y\ell} m$ vertices of $\bigcup_{s=1}^{x+1} V_{i_s,j_s}$, then we can find a $K_r$-tiling covering all but at most $\frac{r\delta}{4}m$ vertices of $V_1\cup V_2\cup\cdots\cup V_r$. Applying this to all copies of $H$ from $\mathcal{H}$ would give us a $K_{r}$-tiling in $G$ covering all but at most
\[
|V_0|+\tfrac{\delta}{2} k m+ |\mathcal{H}|\tfrac{r\delta}{4}m< \eps  n+ \tfrac{\delta}{2}n+ \tfrac{\delta}{4}n< \delta n
\]
vertices. So to complete the proof of Lemma \ref{tiling lem}, it is sufficient to prove the following claim.

\bclaim\label{embedding lem}
Given any legal family $\{V_{i_s,j_s}\}_{s=1}^{x+1}$, $G[\bigcup_{s=1}^{x+1} V_{i_s,j_s}]$ admits a $K_r$-tiling covering all but at most $\frac{r\delta}{4y\ell} m$ vertices of $\bigcup_{s=1}^{x+1} V_{i_s,j_s}$.
\eclaim
\bproof[Proof of claim]\renewcommand\qedsymbol{$\blacksquare$}
For convenience,
we write $Y_s:=V_{i_{s},j_{s}}$ with $s\in [x+1]$. Recall that $|Y_1|=\frac{m}{\ell}$ and $|Y_s|=\frac{m}{y}$ for $s\in[2,x+1]$. If we can greedily pick vertex-disjoint copies of $K_r$ such that each contains exactly $y$ vertices in $Y_{1}$ and $\ell$ vertices in $Y_{s}$ for each $s\in[2,x+1]$, then almost all vertices in $\cup_{s=1}^{x+1} Y_s$ can be covered in this way. Now it suffices to show that for any $Y_s'\subseteq Y_s$ with $s\in[x+1]$, each of size at least $\frac{\delta}{4y\ell} m$, there exists a copy of $K_r$ with exactly $y$ vertices inside $Y_1'$ and $\ell$ vertices inside each $Y_s'$.

For any distinct $s,t\in[x+1]$, the pair $(V_{i_s},V_{i_t})$ is $\eps$-regular with density at least $d$. Then Fact~\ref{slicing lem} implies that every two sets from $Y_1',\ldots,Y_{x+1}'$ form an $\eps'$-regular pair with density at least $d-\eps$, where $\eps'=\frac{4y\ell}{\delta}\eps$.
Therefore as $\frac{1}{n}\ll \eps\ll \delta,\frac{1}{r}$, by applying Lemma~\ref{hu} on $G$ with $V_i=Y_i', p=\ell,q=x+1, \eta=\frac{\delta}{4y\ell} \frac{m}{n}$ and the fact that $f(n)\le g(n)$, we obtain a copy of $K_{(x+1)\ell}$ which contains exactly $\ell$ vertices in each $Y_s'$ for $s\in [x+1]$. Thus we obtain a desired copy of $K_r$ by discarding arbitrary $\ell-y$ vertices from $Y_1'$ from the clique above.
\eproof
\eproof

\section{Building an absorbing set}\label{sec4}



The construction of an absorbing set is now known via a novel idea of Montgomery~\cite{Montgomery}, provided that (almost) every set of $h$ vertices has linearly many vertex-disjoint absorbers as aforementioned.
Such an approach is summarized as the following result by Nenadov and Pehova~\cite{Nenadov2018}.

\blemma\emph{\cite{Nenadov2018}}\label{Nenadov lem}
Let $H$ be a graph with $h$ vertices and let $\gamma > 0$ and $t \in \mathbb{N}$ be constants.  Then there exist $\xi = \xi(h,t,\gamma)$ and $n_0\in \mathbb{N}$ such that the following statement holds. Suppose that $G$ is a graph with $n \ge n_{0}$ vertices such that every $S \in \tbinom{V(G)}{h} $ has a family of at least $\gamma n$ vertex-disjoint $(S,t)$-absorbers. Then $G$ contains a $\xi$-absorbing set of size at most $\gamma n$.
\elemma

\subsection{Finding absorbers}

In order to find linearly many vertex-disjoint absorbers for (almost) every $h$-subset, we shall use a notion of reachability introduced by Lo and Markstr\"om~\cite{2015Lo}. Here we introduce a slightly different version in our setup.
Let $G$ be a graph of $n$ vertices and $H$ be a graph of $h$ vertices. For any two vertices $u, v \in V(G)$, a set $S\subset V(G)$ is called an \emph{$H$-reachable set} for $\{u, v\}$ if both $G[\{u\} \cup S]$ and $G[\{v\} \cup S]$ have $H$-factors. For $ t \ge 1$ and $\beta > 0$, we say that two vertices $u$ and $v$ are \emph{$(H,\beta,t)$-reachable} (in $G$) if there are $\beta n$ vertex-disjoint $H$-reachable sets $S$ in $G$, each of size at most $ht-1$. Moreover, we say that a vertex set $U \subseteq V(G)$ is \emph{$(H,\beta,t)$-closed} if any two vertices in $U$ are $(H,\beta,t)$-reachable in $G$. Note that the corresponding $H$-reachable sets for $u,v$ may not be included in $U$. We say $U$ is $(H,\beta,t)$-\emph{inner-closed} if $U$ is $(H,\beta,t)$-\emph{closed} and additionally the corresponding $H$-reachable sets for every pair $u,v$ also lie inside $U$.

The following result from~\cite{han2021ramseyturan} builds a sufficient condition to ensure that every $h$-subset $S$ has linearly many vertex-disjoint absorbers.

\blemma[\cite{han2021ramseyturan}]\label{absorbers}
Given $h,t\in\mathbb{N}$ with $h\ge3$ and $\beta >0$, the following holds for any $h$-vertex graph $H$ and sufficiently large $n\in\mathbb{N}$. Let $G$ be an $n$-vertex graph such that $V(G)$ is $(H,\beta,t)$-closed. Then every $S \in \tbinom{V(G)}{h}$ has a family of at least $\frac{\beta}{h^3t} n$ vertex-disjoint $(S,t)$-absorbers.
\elemma

Based on this lemma, it suffices to show that $V(G)$ is closed. However, we shall show a slightly weaker result, namely, there exists a small vertex set $B$ such that the induced subgraph $G-B$ is inner-closed. The proof of Lemma~\ref{closed} can be found in Section~\ref{sec4.2}.

\blemma\label{closed}
Given $r,\ell\in\mathbb{N}$ with $r > \ell \ge 2$ and $\tau, \mu$ with $0<\tau<\mu$, there exist $\alpha, \beta>0$ such that the following holds for sufficiently large $n\in \mathbb{N}$. Let $G$ be an $n$-vertex graph with \[\delta(G) \ge \max\left\{\frac{r-\ell}{r}+\mu,  \frac{1}{2-\varrho^*_{\ell}(r-1,f)} + \mu \right\}n ~\text{and}~\alpha_{\ell}(G) < f(\alpha n).\] Then $G$ admits a partition $V(G)=B\cup U$ such that $|B|\le \tau n$ and $U$ is $(K_r,\beta,4)$-inner-closed.
\elemma


Then by Lemma~\ref{absorbers} applied on $G[U]$, we can easily get the following corollary.

\bcorollary\label{Coro}

Given positive integers $r,\ell$ with $r>\ell\ge2$ and $\tau, \mu$ with $0<\tau<\mu$, there exist $0<\alpha<\beta <\mu^3$ such that the following holds for sufficiently large $n\in \mathbb{N}$. Let $G$ be an $n$-vertex graph with $\delta(G) \ge \left( \frac{1}{2-\varrho^*_{\ell}(r-1,f)} + \mu \right)n $ and $\alpha_{\ell}(G) < f(\alpha n)$. Then $G$ admits a partition $V(G)=B\cup U$ such that $|B|\le \tau n$ and every $S \in \tbinom{U}{r}$ has a family of at least $\frac{\beta}{4r^3} n$ vertex-disjoint $(S,4)$-absorbers in $U$.

\ecorollary

To deal with the exceptional vertex set $B$, we shall pick mutually vertex-disjoint copies of $K_r$ each containing a vertex in $B$. To achieve this, one has to make sure that every vertex $v\in V(G)$ is covered by many copies of $K_r$ in $G$ (the aforementioned \emph{cover threshold}). The following result enables us to find linearly many copies of $K_r$ covering any given vertex.

\begin{proposition}\label{coverthreshold}
Given $r,\ell \in \mathbb{N}$ and a constant $\mu >0$, there exists $\alpha >0$ such that for all sufficiently large $n$ the following holds. Let $G$ be an $n$-vertex graph with $\delta(G) \ge \left( \frac{1}{2-\varrho^*_{\ell}(r-1,f)} + \mu \right)n$ and $\alpha_{\ell}(G)\le f(\alpha n)$. If $W$ is a subset of $V(G)$ with $|W|\le\frac{\mu}{2}n$, then for each vertex $u\in V(G)\setminus W$, $G[V(G)\setminus W]$ contains a copy of $K_r$ covering $u$.
\end{proposition}
\bproof
 We choose $\frac{1}{n}\ll \alpha\ll \mu,\frac{1}{\ell}$ and let $G_1:=G[V(G)\setminus W]$. It suffices to show that for each vertex $u\in V(G_1)$, there is a copy of $K_{r-1}$ in $N_{G_1}(u)$. Note that for every vertex $u$ in $G_1$, we have $|N_{G_1}(u)|\ge \delta(G_1)\ge\delta(G)-|W|\ge\left( \frac{1}{2-\varrho^*_{\ell}(r-1,f)} + \frac{\mu}{2} \right)n$. Given any vertex $v\in N_{G_1}(u)$ with $d_{G_1}(u,v):=|N_{G_1}(u)\cap N_{G_1}(v)|$, we have
\[
  \begin{split}
d_{G_1}(u,v)-\left(\varrho^*_{\ell}(r-1,f)+\tfrac{\mu}{4} \right)d_{G_1}(u)&\ge d_{G_1}(u)+d_{G_1}(v)-n-\left(\varrho^*_{\ell}(r-1,f)+\tfrac{\mu}{4} \right)d_{G_1}(u)\\
&\ge\left(2-\varrho^*_{\ell}(r-1,f)-\tfrac{\mu}{4} \right)\delta(G_1)-n>\tfrac{\mu}{8}n >0.
 \end{split}
 \]
Thus $\delta(G[N_{G_1}(u)])>(\varrho^*_{\ell}(r-1,f)+\frac{\mu}{4})|N_{G_1}(u)|$. Therefore by the definition of $\varrho^*_{\ell}(r-1,f)$ and the choice that $\frac{1}{n}\ll \alpha\ll \mu$, $G[N_{G_1}(u)]$ contains a copy of $K_{r-1}$, which together with $u$ yields a copy of $K_r$ in $G_1$.
\eproof

Now we are ready to prove Lemma~\ref{absorbing lem} using Corollary~\ref{Coro} and Proposition~\ref{coverthreshold}.

\bproof[Proof of Lemma~\ref{absorbing lem}]
Given positive integers $\ell,r$ with $r>\ell\ge2$ and $\mu,\gamma$ with $0<\gamma\le\frac{\mu}{2}$, we choose $\frac{1}{n}\ll\alpha\ll\xi\ll\beta,\tau\ll \gamma,\frac{1}{r}$.  Let $G$ be an $n$-vertex graph with $\delta(G) \ge \left( \frac{1}{2-\varrho^*_{\ell}(r-1,f)} + \mu \right)n$, $\alpha_{\ell}(G) < f(\alpha n)$ and $n\in r\mathbb{N}$. Then Corollary \ref{Coro} implies that $G$ admits a partition $V(G)=B \cup U$ such that $|B|\le \tau n$ and every $S \in \tbinom{U}{r}$ has a family of at least $\frac{\beta}{4r^3} n$ vertex-disjoint $(S,4)$-absorbers in $U$. Let $G_1:=G[U]$. Then by applying Lemma \ref{Nenadov lem} on $G_1$, we obtain in $G_1$ a $\xi$-absorbing subset $A_1$ of size at most $\frac{\beta}{4r^3} n$.

Now, we shall iteratively pick vertex-disjoint copies of $K_r$ each covering at least a vertex in $B$ whilst avoiding using any vertex in $A_1$, and we claim that every vertex in $B$ can be covered in this way.

Let $G_2:=G-A_1$.
For $u\in B$, we apply Proposition~\ref{coverthreshold} iteratively to find a copy of $K_r$ covering $u$ in $G_2$, while avoiding $A_1\cup B$ and all copies of $K_r$ found so far.
Because of the fact that $\beta,\tau\ll\gamma,\frac{1}{r}$, this is possible as during the process, the number of vertices that we need to avoid is at most $|A_1|+r|B|\le \frac{\beta}{4r^3}n + r \tau n\le\frac{\mu}{2}n$.
Let $K$ be the union of the vertex sets over all copies of $K_r$ covering $B$ and $A:=A_1\cup K$.
Recall that $A_1$ is a $\xi$-absorbing set for $G_1=G-B$, and $B\subseteq K\subseteq A$. Then it is easy to check that $A$ is a $\xi$-absorbing set for $G$ and
\[
|A|=|A_1|+|K|\le \tfrac{\beta}{4r^3} n+r\tau n\le \gamma n,
\]
where the last inequality follows since $\beta,\tau\ll\gamma, \frac{1}{r}$.
\eproof
Now it remains to prove Lemma~\ref{closed}, which is done in the next subsection.
\subsection{Proof of Lemma~\ref{closed}}\label{sec4.2}
The proof of Lemma~\ref{closed} makes use of Szemer\'{e}di's Regularity Lemma and a result in~\cite{chang2021}.
\blemma\emph{\cite[Lemma 5.1]{chang2021}}\label{reachable lem}
Given $n,r,\ell\in\mathbb{N}$ with $r > \ell \ge 2$ and a monotone increasing function $f(n)$, for all $\tau, \mu$ with $0<\tau<\mu$, there exist positive constants $\beta_1, \gamma_1$ and $\alpha>0$ such that the following holds for sufficiently large $n\in \mathbb{N}$. Let $G$ be an $n$-vertex graph with $\delta(G) \ge \left( \frac{1}{2-\varrho^*_{\ell}(r-1,f)} + \mu \right)n $ and $\alpha_{\ell}(G) \le f(\alpha n)$. Then $G$ admits a partition $V(G)=B\cup U$ such that $|B|\le \tau n$ and every vertex in $U$ is $(K_r,\beta_1,1)$-reachable to at least $\gamma_1n$ other vertices in $U$ with all the corresponding $K_r$-reachable sets belonging to $U$.
\elemma


\begin{proof}[Proof of Lemma~\ref{closed}]
Given $r,\ell\in\mathbb{N}$ with $r > \ell \ge 2$ and $\tau, \mu$ with $0<\tau<\mu$, we choose constants \[\tfrac{1}{n}\ll\alpha\ll\beta\ll\eps\ll \beta_1,\gamma_1\ll \tau,\mu,\tfrac{1}{\ell}\] and let $G$ be an $n$-vertex graph with $\delta(G) \ge \max\{\frac{r-\ell}{r}+\mu,  \frac{1}{2-\varrho^*_{\ell}(r-1,f)} + \mu \}n $ and $\alpha_{\ell}(G) < f(\alpha n)$. Then by applying Lemma~\ref{reachable lem}, we obtain a partition $V(G)=B\cup U$ such that $|B|\le \tau n$ and every vertex in $U$ is $(K_r,\beta_1,1)$-reachable $({\rm in} \ G[U])$ to at least $\gamma_1n$ other vertices in $U$. For any two vertices $u,v\in U$, we shall prove in $G[U]$ that $u,v$ are $(K_r, \beta,4)$-reachable.

Let $X$ and $Y$ be the sets of vertices that are $(K_r, \beta_1, 1)$-reachable to $u$ and $v$, respectively. By taking subsets from them and renaming if necessary, we may further assume that $X\cap Y=\emptyset$ and $|X|=|Y|=\frac{\gamma_1n}{2}$. Then by applying Lemma~\ref{reg} on $G$ with positive constants $\eps\ll \beta_1,\gamma_1$ and $d:=\frac{\mu}{4}$, we obtain a refinement $\mathcal{P}:=\{V_0,V_{i},\ldots,V_k\}$ of the original partition $\{X,Y,V(G)-X-Y\}$ and a spanning subgraph $G'\subseteq G$ with properties (a)-(e), where we let $m:=|V_i|$ for all $i\in[k]$ and $R_d$ be the corresponding $d$-reduced graph.
Without loss of generality, we may assume that $V_1\subseteq X$ and $V_2\subseteq Y$. Note that by Fact~\ref{min degree}, we can observe that $\delta(R_d)\ge \max\{\frac{r-\ell}{r}+\frac{\mu}{2},  \frac{1}{2-\varrho^*_{\ell}(r-1,f)} + \frac{\mu}{2} \}k\ge(\frac{1}{2}+\frac{\mu}{2})k$. Let $V_3$ be a common neighbor of $V_1$ and $V_2$ in $R_d$.

Now we shall show that $u$ and $v$ are $(K_r, \beta,4)$-reachable. We write $r=\ell x+y$ for some integers $x,y$ with $x>0,0\le y\le \ell-1$. Note that $\delta(R_d)\ge (\frac{r-\ell}{r}+\frac{\mu}{2})k\ge (\frac{x-1}{x}+\frac{\mu}{2})k$. Thus every $x$ vertices in $R_d$ have at least $\frac{\mu}{2}xk$ common neighbors and we can greedily pick two copies of $K_{x+1}$ in $R_d$ that contain the edge $V_1V_3$ and $V_2V_3$ respectively and overlap only on the vertex $V_3$. We use $\mathcal{A}=\{V_1,V_3, V_{a_1},V_{a_2},\ldots,V_{a_{x-1}}\}$ and $\mathcal{T}=\{V_2,V_3, V_{b_1},V_{b_2},\ldots,V_{b_{x-1}}\}$ to denote the two family of clusters related to the two copies of $K_{x+1}$ in $R_d$. Applying Lemma~\ref{hu} on $G$ with $\eta=\frac{1}{2}\frac{m}{n},q=x+1,p=\ell$ and $V_1,V_3, V_{a_1},\ldots,V_{a_{x-1}}$ playing the role of $V_1,\ldots,V_q$, we can iteratively take $\frac{m}{2\ell}$ vertex-disjoint copies of $K_{r+1}$ (since $r+1\le pq$) which are denoted by $S_1,S_2,\ldots,S_{\frac{m}{2\ell}}$, such that each $S_i$ has exactly $y+1$ vertices in $V_3$, $\ell$ vertices in $V_1$ and $V_{a_i}$, $i\in[x-1]$. Let $V_3'\subseteq V_3$ be a subset obtained by taking exactly one vertex from each such $S_i$. Then $|V_3'|=\frac{m}{2\ell}$ and again by applying Lemma~\ref{hu} on $G$ with $\eta=\frac{1}{4\ell}\frac{m}{n}, q=x+1,p=\ell$ and $V_2,V_3', V_{b_1},\ldots,V_{b_{x-1}}$ playing the role of $V_1,\ldots,V_q$, we can greedily pick $\frac{m}{4\ell^2}$ vertex-disjoint copies of $K_{r+1}$, denoted by $T_1,T_2,\ldots,T_{\frac{m}{4\ell^2}}$, such that each $T_i$ has exactly $y+1$ vertices in $V_3'$, $\ell$ vertices in $V_2$ and $V_{b_i}$, $i\in[x-1]$.

\begin{figure}[htb]
\center{\includegraphics[width=10.03cm] {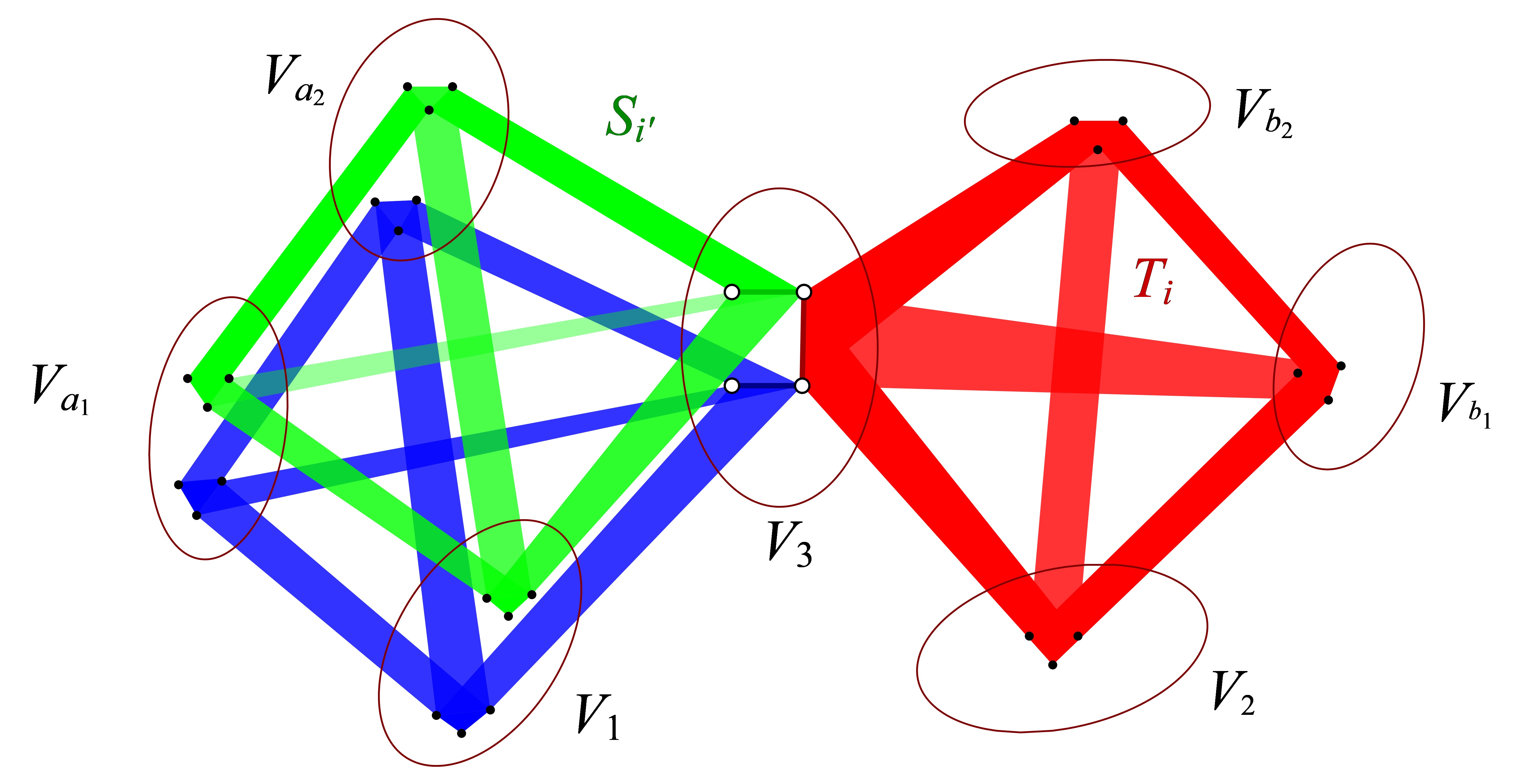}}
\caption{$S_i$ and $T_i$: here we take $\ell=3,y=1$ for instance.}
\label{figure1}
\end{figure}
Now it remains to show the following statement with $\beta n\le\frac{m}{4\ell^2}$. Recall that $m\le \eps n$.
\begin{claim}
There exist $\beta n$ vertex-disjoint $K_r$-reachable sets for $u,v$, each of size $4r-1$.
\end{claim}
\begin{figure}[htb]
\center{\includegraphics[width=10.4cm] {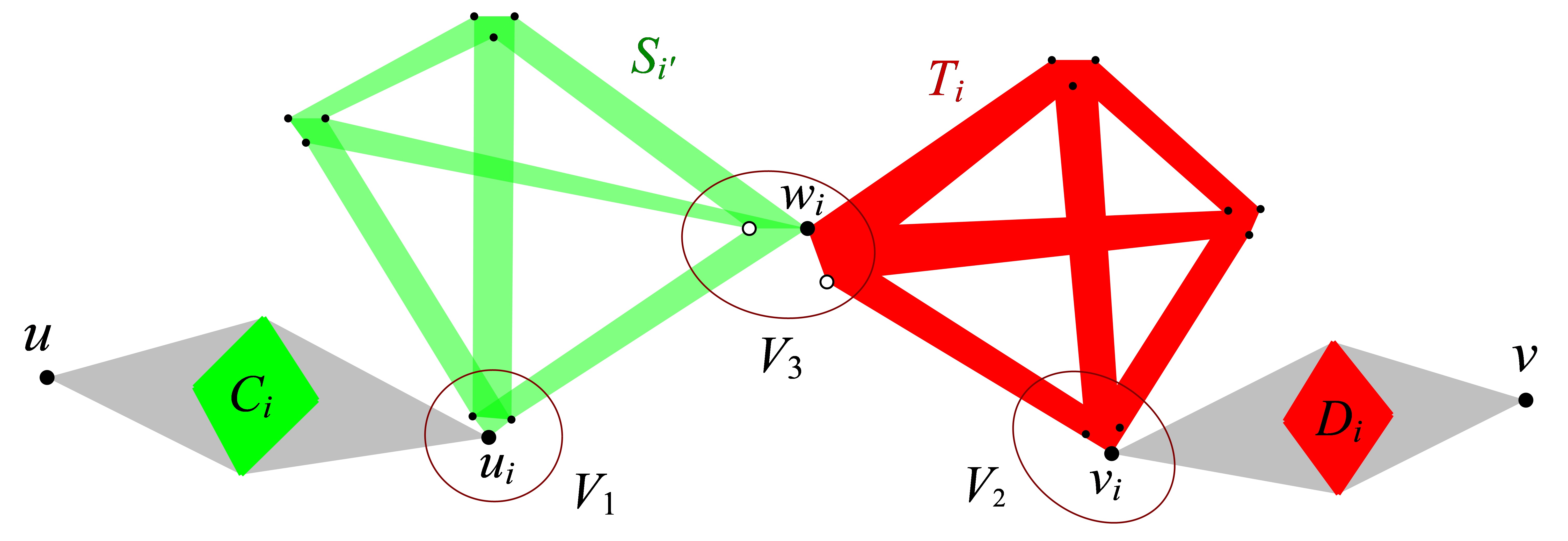}}
\caption{Constructions of $K_r$-reachable sets $E_i$.}
\label{figure2}
\end{figure}
\bproof[Proof of claim]\renewcommand\qedsymbol{$\blacksquare$}
Here the main idea is to extend all such $T_i$'s to pairwise vertex-disjoint $K_r$-reachable sets. Note that for each $T_i$, there exists $S_{i'}$ such that the two copies of $K_{r+1}$ intersect on exactly one vertex in $V_3'$, denoted by $w_i$. Let $u_i$ be an arbitray vertex chosen from $S_{i'}$ that lies in $V_1$, and $v_i$ be chosen from $T_i$ that lies in $V_2$.
Then by the assumption that $u$ is $(K_r, \beta_1, 1)$-reachable to $u_1$, there exist at least $\beta_1 n$ vertex-disjoint $K_r$-reachable sets for $u$ and $u_i$ (resp.~$v$ and $v_i$). Therefore by the fact that $\eps\ll \beta_1$, we can greedily choose two vertex-disjoint $K_r$-reachable sets, say $C_i$ and $D_i$, for $u,u_i$ and $v,v_i$, respectively, which are also disjoint from all cliques $S_{i}$ or $T_j$ for $i\in[\frac{m}{2\ell}],j\in[\frac{m}{4\ell^2}]$. It is easy to check that the set \[E_i:=V(S_{i'})\cup V(T_i)\cup C_i\cup D_i\] has size $4r-1$ and $G[E_i\cup\{u\}]$ (similarly for $E_i\cup\{v\}$) contains $4$ copies of $K_r$, which are induced on the sets $\{u\}\cup C_i, V(S_{i'})-\{w_i\}, V(T_i)-\{v_i\}$ and $\{v_i\}\cup D_i$, respectively. Thus by definition $E_i$ is a $K_r$-reachable set for $u$ and $v$. A desired number of mutually disjoint $K_r$-reachable set $E_i$ can be chosen by extending each $T_i$ as above.
\eproof

\end{proof}

\section{A construction}

In this section, we shall use a construction of $K_{\ell+1}$-free graphs $G$ with small $\alpha_{\ell}(G)$ to prove Proposition~\ref{lower}. An explicit construction was firstly obtained by Erd\H{o}s and Rogers \cite{erdos-rogers} in the setting that $\alpha_{\ell}(G)=o(n)$. Here we give a probabilistic construction as follows, whose proof is very similar to that of a result of Nenadov and Pehova (see Proposition 4.1 in \cite{Nenadov2018}).

\begin{lemma}\label{G2}
For any $\ell\in \mathbb{N}$ with $\ell\ge2$, a constant $\gamma\in(0,\frac{\ell-1}{\ell^2+2\ell})$ and sufficiently large integer $n$, there exists an $n$-vertex $K_{\ell+1}$-free graph $G_{\ell}$ such that $\alpha_{\ell}(G_{\ell})\le n^{1-\gamma}$.
\end{lemma}

Now we firstly give a short proof of Proposition~\ref{lower}, and then present the proof of Lemma~\ref{G2} at the end of this section.
\begin{proof}[Proof of Proposition~\ref{lower}]
Fix $r>\ell\ge2$ and constants $\eta,\gamma$ as in the statement. Let $n$ be sufficiently large and define $\mu=\tfrac{r}{r-\ell}(\tfrac{r-\ell}{r}-\eta)>0$. Then by Lemma~\ref{G2}, we choose $G_{\ell}$ to be a $(1-\eta)n$-vertex $K_{\ell+1}$-free graph with $\alpha_{\ell}(G_{\ell})\le |G_{\ell}|^{1-\gamma}$.

Let $G$ be an $n$-vertex graph with vertex partition $V(G)=X_1\cup X_2$ such that
\begin{enumerate}
  \item[(i)] $G[X_1]$ is a clique with $|X_1|=\eta n$;
  \item[(ii)] $G[X_1, X_2]$ is a complete bipartite graph;
  \item[(iii)] $G[X_2]$ induces a copy of $G_{\ell}$.
\end{enumerate} Now we claim that $G$ has the desired properties.
Indeed it is easy to see that $\de(G)\ge  \eta n$ and $\alpha_{\ell}(G)\le \alpha_{\ell}(G_{\ell})< n^{1-\gamma}$. Since $G[X_2]$ is $K_{\ell+1}$-free, every copy of $K_r$ must intersect $X_1$ on at least $r-\ell$ vertices. Thus every $K_r$-tiling in $G$ contains at most $\tfrac{|X_1|}{r-\ell}=\tfrac{\eta}{r-\ell} n$ vertex-disjoint copies of $K_r$, which together cover at most $\tfrac{r\eta}{r-\ell} n=(1-\mu)n$ vertices.
\end{proof}

\begin{proof}[Proof of Lemma~\ref{G2}]
We only consider $\ell\ge 3$ as the case $\ell=2$ follows from a celebrated result on Ramsey number that $R(3,n)=\Theta(n^2/\log n)$~\cite{kim95}.
We choose $\frac{1}{n}\ll\gamma, \frac{1}{\ell}$ and let $x=\tfrac{2-\gamma}{\ell+1}$. Considering the random graph $G=G(n,p)$ with $p=n^{-x}$, we shall verify that with positive probability, $G$ is $K_{\ell+1}$-free and $\alpha_{\ell}(G)\le n^{1-\gamma}$.
By applying the FKG inequality~\cite{FKG}, we have that

\begin{align}
  \mathbb{P}[\text{$G$ is $K_{\ell+1}$-free}]&\ge \prod_{S\in\binom{V(G)}{\ell+1}}\mathbb{P}[G[S]\neq K_{\ell+1}]  \nonumber \\
  & \ge\left(1-p^{\binom{\ell+1}{2}}\right)^{\binom{n}{\ell+1}}\ge \exp\left(-2p^{\binom{\ell+1}{2}}n^{\ell+1}\right)=\exp\left(-2n^{1+\frac{\gamma}{2}\ell}\right), \nonumber
\end{align}
where we bound $1-x\ge e^{-2x}$ for $x\in (0,\frac{1}{2})$. Now it remains to determine the probability of the event that $\alpha_{\ell}(G)\le n^{1-\gamma}$.

Let $I$ be the random variable counting all sets $A$ such that $|A|= n^{1-\gamma}$ and $G[A]$ is $K_{\ell}$-free. Then \[\mathbb{E}(I)= \sum_{|A|=n^{1-\gamma}}\mathbb{P}[\text{$G[A]$ is $K_{\ell}$-free}].\]
Here we shall use a powerful inequality of Janson \cite{janson}, where for each $\ell$-set $S\subseteq A$ we denote by $X_S$ the indicator variable for the event that $G[S]= K_{\ell}$. Let $X=\sum_{S\in\binom{A}{\ell}} X_S$. Then by Janson's inequality, we obtain that \[\mathbb{P}[\text{$G[A]$ is $K_{\ell}$-free}]=\mathbb{P}[X=0]\le\exp\left(-\mathbb{E}(X)+\frac{\Delta}{2}\right),\]
where $\mathbb{E}(X)=\binom{|A|}{\ell}p^{\binom{\ell}{2}}=\Theta\left(n^{(1-\gamma)\ell-\frac{2-\gamma}{\ell+1}\binom{\ell}{2}}\right)$ and $\Delta=\sum\limits_{S\neq S',~|S\cap S'|\ge2}\mathbb{P}[X_S=1,X_{S'}=1]$. Note that
\begin{align}
\Delta&=\sum\limits_{S\neq S',~|S\cap S'|\ge2}\mathbb{P}[X_S=1,X_{S'}=1] \nonumber \\
&\le \binom{|A|}{\ell}p^{\binom{\ell}{2}}\sum_{2\le s\le \ell-1}\binom{\ell}{s}\binom{|A|-\ell}{\ell-s}p^{\binom{\ell}{2}-\binom{s}{2}} \nonumber \\
&\le \mathbb{E}(X)\sum_{2\le s\le \ell-1}\binom{\ell}{s}n^{(\ell-s)(1-\gamma-x\frac{\ell+s-1}{2})} \nonumber \\
&=o(\mathbb{E}(X)),  \nonumber
\end{align}
where the last equality follows because $x=\frac{2-\gamma}{\ell+1}$ and thus $1-\gamma-x\frac{\ell+s-1}{2}\le-\frac{\gamma}{2}$ holds for any $s\ge 2$.
Therefore $\mathbb{E}(I)\le 2^n\exp\left(-\Theta\left(n^{(1-\gamma)\ell-\frac{2-\gamma}{\ell+1}\binom{\ell}{2}}\right)\right)$ and by Markov's inequality, with probability at least $1-2^n\exp\left(-\Theta\left(n^{(1-\gamma)\ell-\frac{2-\gamma}{\ell+1}\binom{\ell}{2}}\right)\right)$, we have $I=0$, that is, $\alpha_{\ell}(G)\le n^{1-\gamma}$.

By the inclusive-exclusive principle, the probability of the event that $G$ is $K_{\ell+1}$-free and $\alpha_{\ell}(G)\le n^{1-\gamma}$ is at least \[\exp\left(-2n^{1+\frac{\gamma}{2}\ell}\right)-2^n\exp\left(-\Theta\left(n^{(1-\gamma)\ell-\frac{2-\gamma}{\ell+1}\binom{\ell}{2}}\right)\right)\] and it is positive for sufficiently large $n$ as long as \[1+\frac{\gamma}{2}\ell<(1-\gamma)\ell-\frac{2-\gamma}{\ell+1}\binom{\ell}{2},\]
which follows easily as $\gamma<\frac{\ell-1}{\ell^2+2\ell}$.
\end{proof}

\section{Concluding remarks}\label{sec7}
In this paper we study the minimum degree condition for $K_r$-factors in graphs with sublinear $\ell$-independence number.
Our result is asymptotically sharp when $\alpha_{\ell}(G)\in (n^{1-\gamma},n^{1-\omega(n)\log^{-\lambda}n})$ for any constant $0<\gamma<\frac{\ell-1}{\ell^2+2\ell}$.


This leads to the following question: What is the general behavior of the minimum degree condition forcing a clique factor when the condition of $\ell$-independence number is imposed within the range $(1,n)$?  We formulate this as follows. Given integers $n>r>\ell\ge 2$ with $n\in r\mathbb{N}$, a constant $\alpha >0$ and a monotone increasing function $g(n)\in[n]$, we denote by $\textbf{RTT}_{\ell}(n,K_{r},g(\alpha n))$ the maximum integer $\delta$ such that there exists an $n$-vertex graph $G$ with $\delta(G) \ge \delta$ and $\alpha_{\ell}(G)\le g(\alpha n)$ which does not contain a $K_{r}$-factor. Here we try to understand when and how the value $\textbf{RTT}_{\ell}(n,K_{r},g(\alpha n))$ changes sharply when the magnitude of $g(n)$ varies. This can be seen as a degree version of the well-known phase transition problem for $\textbf{RT}_{2}(n,K_{r},g(n))$ in Ramsey--Tur\'{a}n theory (see \cite{BALOGH2015148,BaloghRT2012,kkl19}). It is worth noting that many open questions on the phase transition problem of $\textbf{RT}_{2}(n,K_{r},g(n))$ are essentially related to Ramsey theory.

Here we consider the basic case $\textbf{RTT}_2(n,K_{r},g(n))$. 
Recall that Knierim and Su \cite{knierim2019kr} resolved Problem~\ref{p1} for $r\ge 4$ by giving an asymptotically tight minimum degree bound $(1- \frac{2}{r})n + o(n)$. In our context of $g(n)=n$, this can be roughly reformulated as \[\textbf{RTT}_2(n,K_{r},o(n))=\frac{r-2}{r}n+o(n)~ \text{for} ~r\ge 4.\]
Also, for integers $r,\ell$ with $r> \ell\ge \frac{3}{4}r$, Theorem~\ref{main thm} can be stated as \[\textbf{RTT}_{\ell}(n,K_{r},o(n))=\frac{1}{2-\varrho_{\ell}(r-1)}n+o(n).\]
In this paper, our main theorem combined with Proposition~\ref{lower} and the cover threshold implies that for $r> \ell\ge 2,\gamma\in(0,\frac{\ell-1}{\ell^2+2\ell})$ and $n^{1-\gamma}\le f(n)\le n^{1-\omega(n)\log^{-\lambda}n}$,
\[\textbf{RTT}_{\ell}(n,K_{r},f(o(n)))=\max\left\{\frac{r-\ell}{r}n, \frac{1}{2-\varrho^*_{\ell}(r-1,f)}n\right\}+o(n).\] This provides an insight into the general behavior of $\textbf{RTT}_{\ell}(n,K_{r},f(n))$ but the asymptotic behavior of $\textbf{RTT}_{\ell}(n,K_{r},g(n))$ for a general $g(n)$ seems to be out of reach. 
It will be interesting to study the case $g(n)=n^c$ for any constant $c\in(0,1-\tfrac{\ell-1}{\ell^2+2\ell})$.

\bibliographystyle{plain}
\bibliography{ref(3)}

\begin{appendices}

\section{Proof of Lemma~\ref{hu}}
We use the method of dependent random choice to prove Lemma~\ref{hu}.
The method was developed by F\"uredi, Gowers,
Kostochka, R\"odl, Sudakov, and possibly many others.
The next lemma is taken from Alon, Krivelevich and Sudakov~\cite{AKS03}.
Interested readers may check the survey paper on this method by Fox and Sudakov~\cite{dependent11}.

\begin{lemma}[Dependent Random Choice]\emph{\cite{AKS03}}\label{deprc}
Let $a,d,m,n,r$ be positive integers. Let $G=(V,E)$ be a graph with
$n$ vertices and average degree $d=2e(G)/n$. If there is a
positive integer $t$ such that
\begin{equation}\label{drc}
\frac{d^t}{n^{t-1}} - \binom{n}{r}\left(\frac{m}{n}\right)^t \ge a,
\end{equation}
then $G$ contains a subset $U$ of at least $a$ vertices such that every $r$ vertices in $U$ have at least $m$ common neighbors.
\end{lemma}

Conlon, Fox, and Sudakov~\cite{CFS09} extended Lemma \ref{deprc} to hypergraphs. The {\em weight} $w(S)$ of a set $S$ of edges in a hypergraph is the number of vertices in the union of these edges.

\begin{lemma}[Hypergraph Dependent Random Choice]\emph{\cite{CFS09}}\label{hydeprc}
Suppose $s, \Delta$ are positive integers, $\eps, \delta > 0$, and
$G_r = (V_1, \ldots, V_r; E)$ is an $r$-uniform $r$-partite hypergraph
with $|V_1|=\ldots=|V_r| = N$ and at least $\eps N^r$ edges. Then
there exists an $(r-1)$-uniform $(r-1)$-partite hypergraph $G_{r-1}$
on the vertex sets $V_2,\ldots,V_r$ which has at least
$\frac{\eps^s}{2}N^{r-1}$ edges and such that for each nonnegative
integer $w\le (r-1)\Delta$, there are at most $4r\Delta
\eps^{-s}\beta^s w^{r\Delta}r^wN^w$ dangerous sets of edges of
$G_{r-1}$ with weight $w$, where a set $S$ of edges of $G_{r-1}$ is
dangerous if $|S|\le \Delta$ and the number of vertices $v\in V_1$
such that for every edge $e\in S, e+v\in G_r$ is less than $\beta N$.
\end{lemma}

\begin{proof}[Proof of Lemma~\ref{hu}]
Given a constant $d>0$ and integers $p,q\ge 2$, we choose \[\tfrac{1}{n}\ll\tfrac{1}{C},\eps\ll d, \tfrac{1}{p}, \tfrac{1}{q}~\text{and in addition}~\tfrac{1}{n}\ll\eta.\] Let $G$ be an $n$-vertex graph with $\alpha_p(G)< g(n)$, where \[g(n)=n2^{-C\log^{1-1/q}n}.\]
Let $V_1,V_2,\ldots,V_{q}$ be given such that $|V_i|\ge \eta n$, $i\in[q]$ and every pair $(V_i,V_j)$ is $\eps$-regular with density at least $d$. We define a $q$-uniform $q$-partite hypergraph $H^0$ whose vertex set is $\cup_{i\in[q]} V_i$ and edge set $E(H^0)$ is the family of $q$-sets that span $q$-cliques in $G$ and contain one vertex from each of $V_1,\ldots,V_q$. We may assume $|V_i| = \eta n=: N$, then by the counting lemma, $|E(H^0)|\ge \eps_0 N^q$, where $\eps_0 > \left( d/3\right)^{\binom{q}{2}}$.
Let \[\beta = \tfrac{g(n)}{N},\quad
s =\log^{\frac{1}{q}}n, \quad \eps_i = \eps_0^{s^i}2^{-\frac{s^i-1}{s-1}}, \quad r_i = q-i, \quad \Delta_i = p^{r_i} \quad \textnormal{and}\quad w_i = pr_i.\]
We start from $H^0$. For $1\le i \le q-2$ we  apply Lemma \ref{hydeprc} to $H^{i-1}$ with $\Delta = \Delta_i, \eps = \eps_{i-1}, r = r_{i-1}$ and $w = w_i$ to get $H^i$.
Note that $\Delta, \eps_0, r, w$ are all constants and $\frac{1}{C}\ll d, \frac{1}{p}, \frac{1}{q}$.
It is easy to check that for $1\le i\le q-2$, we have
\begin{eqnarray*}
4r\Delta\eps^{-s}\beta^s w^{r\Delta}r^{w}N^{w}
 &=&O\left(2^{2\log^{\frac{i-1}{q}}n}\eps_0^{-\log^{\frac{i}{q} }n}(1/\eta)^{\log^{\frac{1}{q}}n }2^{-C\log n} (\eta n)^{w}\right)\\
 &=&O(n^{-C/2})
 = o(1) < 1.
\end{eqnarray*}
Then by Lemma \ref{hydeprc} there exists an $r_i$-uniform
$r_i$-partite hypergraph $H^i$ on the vertex sets $V_{i+1},\ldots,V_q$
that contains at least $\eps_i N^{r_i}$ edges and contains no
dangerous sets of $\Delta_i$ edges on $w_i$ vertices. (Recall that a set $S$ of $\Delta_i$ edges on $w_i$ vertices is dangerous if the number of vertices $v \in V_i$ for which for every edge $e\in S, e+v \in H^{i-1}$ is less than $\beta N$). Now we have a
hypergraph sequence $\{H^\ell\}_{\ell=0}^{q-2}$.  We will prove by
induction on $i$ that there is a $p$-set $A^{q-\ell}\subset V_{q-\ell}$ for
$0\le \ell\le i$ such that $G\left[A^{q-\ell}\right]=K_p$ and
$H^{q-i-1}\left[\bigcup_{\ell=0}^{i} A^{q-\ell}\right]$ is complete $r_{q-i-1}$-partite. Note that if a vertex set $T$ is an edge
of $H^0$, then $G[T]$ is a $q$-clique. So
$G\left[\bigcup_{\ell=0}^{q-1} A^{q-\ell}\right] = K_{pq}$, which
will prove Lemma~\ref{hu}.

We first show that the induction hypothesis holds for $i=1$. Note that $r_{q-2}$ = 2, so $H^{q-2}$ is a bipartite graph on $2N$ vertices with at least $\eps_{q-2}N^2$ edges. We now apply Lemma~\ref{deprc} to $H^{q-2}$ with
$$a=2\beta N,  \qquad d=\eps_{q-2}N, \qquad t=s, 
  \qquad r=p \qquad \textnormal{and} \qquad m=\beta N.$$
We check condition \eqref{drc}:
\begin{eqnarray*}
\frac{(\eps_{q-2}N)^s}{(2N)^{s-1}}-{\binom{2N}{p}}\left( \frac{\beta N}{2N} \right)^s
&\ge& (\eps_0/2)^{\log^{1-1/q}n}N - n^p (1/2\eta)^{\log^{\frac{1}{q}}n}2^{-C\log n} \\
&=&  (\eps_0/2)^{\log^{1-1/q}n}N - o(1)
\ge 2\beta N,
\end{eqnarray*}
where the last equality and inequality follow as long as $C>\max\{p,\log \frac{2}{\eps_0}\}$. Therefore we have a subset $U$ of $V_{q-1}\cup V_{q}$ with $|U| = 2\beta N$ such that
every $p$ vertices in $U$ have at least $\beta N$ common neighbors in $H^{q-2}$. Either $V_{q-1}$ or $V_{q}$ contains at least half of the vertices of $U$, so w.l.o.g. we may assume
that $U' = U\cap V_{q-1}$ contains at least $\beta N = m$ vertices.
Because $\alpha_p(G) < m$, the vertex set $U'$ contains a $p$-vertex set $A^{q-1}$
such that $G\left[A^{q-1}\right] = K_p$. The vertices of $A^{q-1}$ have at least $m$ common neighbors in $V_q$, so their common neighborhood also contains a $p$-vertex subset $A^q$ of $V_q$ such that $G[A^q] = K_p$.
 Now $H^{q-2}\left[A^{q-1}\cup A^q\right]$ is complete bipartite.
 We are done with the base case $i=1$.

 For the induction step, assume that the induction hypothesis
holds for $i-1$, then we can find a complete $r_{q-i}$-partite subhypergraph  $\widetilde{H}^{q-i}$  of $H^{q-i}$ spanned by $\bigcup_{\ell=0}^{i-1} A^{q-\ell}$, where $G[A^{q-\ell}] = K_p$ for  every $\ell$. The hypergraph
$H^{q-i}$ has no dangerous set of $\Delta_{q-i}$ edges on $w_{q-i}$ vertices, and $\widetilde{H}^{q-i}$ contains $pi = w_{q-i}$ vertices and $p^i = \Delta_{q-i}$ edges,
so $\widetilde{H}^{q-i}$  is not dangerous.
 Then we can find a set $B$ of $\beta N$ vertices in $V_{q-i}$ such that
 for every edge $e\in \widetilde{H}^{q-i}$ and every vertex $v\in B$, $e+v \in H^{q-i-1}$, which means that $ H^{q-i-1}\left[B\cup\bigcup_{\ell=0}^{i-1} A^{q-\ell}\right]$ is complete $r_{q-i-1}$-partite. Then,
 because $\alpha_p(G) < \beta N$, we can find a $p$-vertex subset $A^{q-i}$ of $B$ such that $G[A^{q-i}] = K_p$.
\end{proof}

\end{appendices}

\end{document}